\documentclass[a4paper,10.95pt]{amsart}
\usepackage[utf8]{inputenc}
\usepackage{amsmath}
\usepackage{amssymb}
\usepackage{graphicx, overpic}
\usepackage{caption}
\usepackage{subcaption}
\usepackage{color}

\usepackage{hyperref}

\usepackage[hmargin=1.165in,vmargin=1.22in]{geometry}

\newtheorem{thm}{Theorem}[section]
\newtheorem*{thmNoCTHyp}{Theorem \ref{thm_noCT_hyp}}
\newtheorem*{thmNoCTIFP}{Theorem \ref{thm_noCT_IFP}}
\newtheorem*{thmNormalStructure}{Special Case of Theorem \ref{thm_structure_normal}}
\newtheorem*{thmPurelyAtoroidal}{Theorem \ref{thm:pur_ator}}
\newtheorem{prop}[thm]{Proposition}
\newtheorem{lemma}[thm]{Lemma}
\newtheorem{cor}[thm]{Corollary}

\theoremstyle{definition}
\newtheorem{defn}[thm]{Definition}
\newtheorem{remark}[thm]{Remark}

\newtheorem{example}[thm]{Example}

\newtheorem{notation}[thm]{Notation}

\newcommand{\field}[1]{\mathbb{#1}}
\newcommand{\Z}{\field{Z}}
\newcommand{\integers}{\field{Z}}

\newcommand{\N}{\field{N}}
\newcommand{\E}{\field{E}}
\newcommand{\Hy}{\field{H}}
\newcommand{\F}{\field{F}}

\newcommand{\cA}{\mathcal{A}}
\newcommand{\cB}{\mathcal{B}}

\newcommand{\cF}{\mathcal{F}}

\newcommand{\cH}{\mathcal{H}}

\newcommand{\cP}{\mathcal{P}}

\newcommand{\la}{\left\langle}
\newcommand{\ra}{\right\rangle}
\newcommand{\p}{\partial}

\newcommand{\G}{\Gamma}
\newcommand{\Lam}{\Lambda}

\newcommand{\hi}{\hat{\iota}}

\DeclareMathOperator{\Isom}{Isom}
\DeclareMathOperator{\CAT}{CAT}

\DeclareMathOperator{\Stab}{Stab}

\DeclareMathOperator{\diam}{diam}
\DeclareMathOperator{\Out}{Out}
\DeclareMathOperator{\Aut}{Aut}
\DeclareMathOperator{\Homeo}{Homeo}

\definecolor{amethyst}{rgb}{0.6, 0.4, 0.8}

\newcommand{\hide}[1]{}
\title{Cannon--Thurston maps for $\CAT(0)$ groups with isolated flats}
\author{Benjamin Beeker, Matthew Cordes, Giles Gardam, Radhika Gupta, and Emily Stark}

\begin{document}

\begin{abstract}
    Mahan Mitra (Mj) proved Cannon--Thurston maps exist for normal hyperbolic subgroups of a hyperbolic group \cite{mitra}. We prove that Cannon--Thurston maps do not exist for infinite normal hyperbolic subgroups of non-hyperbolic $\CAT(0)$ groups with isolated flats with respect to the visual boundaries. We also show Cannon--Thurston maps do not exist for infinite infinite-index normal $\CAT(0)$ subgroups with isolated flats in non-hyperbolic $\CAT(0)$ groups with isolated flats. We obtain a structure theorem for the normal subgroups in these settings and show that outer automorphism groups of hyperbolic groups have no purely atoroidal $\Z^2$ subgroups.
\end{abstract}

\maketitle

\vskip.2in

  \section{Introduction}

  Cannon--Thurston maps, named after the seminal work of Cannon and Thurston in the setting of hyperbolic $3$-manifolds, allow one to combine two far-reaching approaches in geometric group theory: the study of the boundaries of a group and the study of the collection of subgroups of a group. Cannon and Thurston~\cite{cannonthurston} proved if $M$ is a closed fibered hyperbolic $3$-manifold with fiber a closed surface $\Sigma$, then there is a continuous surjective map from the visual boundary of the surface subgroup $\p\widetilde{\Sigma}\cong \p \Hy^2 \cong S^1$ to the visual boundary of the $3$-manifold group $\p \widetilde{M} \cong \p \Hy^3 \cong S^2$ that extends the inclusion $\widetilde{\Sigma} \hookrightarrow \widetilde{M}$ between the universal covers.  
  
  The study of group boundaries is especially fruitful in the presence of hyperbolic or non-positive curvature.
  Gromov~\cite{gromov} proved every hyperbolic group $G$ admits a visual boundary, which is unique up to $G$-equivariant homeomorphism and which yields a compactification of any Cayley graph for $G$ constructed with respect to a finite generating set. A similar theorem was obtained by Hruska--Kleiner~\cite{hruskakleiner} for $\CAT(0)$ groups with isolated flats. $\CAT(0)$ spaces with isolated flats, roughly speaking, are negatively curved away from a collection of isometrically embedded Euclidean flats and share common features with $\delta$-hyperbolic spaces~\cite{hruska,hruskakleiner,haulmark,hruskaruane}. In this paper, we establish a striking distinction between these classes of groups in terms of Cannon--Thurston maps.
  
  The Cannon--Thurston map for hyperbolic groups $H \leq G$ is the continuous map $\p H \rightarrow \p G$ that extends the inclusion from a Cayley graph for $H$ into a Cayley graph for $G$ given by choosing a finite generating set for $H$ and extending it to a finite generating set for $G$, when such a continuous extension exists. Generalizing to include groups acting on $\CAT(0)$ spaces, orbit maps between $\CAT(0)$ model spaces yield a natural notion of inclusion from a subgroup to a group. The Cannon--Thurston map is, again, defined as a continuous $H$-equivariant extension of this inclusion to the visual boundaries of the model spaces. See Section~\ref{sec:CT_map} for details. 
  
  Mitra (Mj) \cite{mitra}, generalizing work of Cannon--Thurston, proved Cannon--Thurston maps exist for normal hyperbolic subgroups of a hyperbolic group. We prove this result does not extend to the setting of $\CAT(0)$ groups with isolated flats, which in this paper have by definition a non-empty collection of flats and are thus non-hyperbolic. 
  
  \begin{thmNoCTHyp}
    Let $G$ be a $\CAT(0)$ group with isolated flats, and let $H \triangleleft G$ be an infinite hyperbolic normal subgroup.    Then the Cannon--Thurston map does not exist for $(H, G)$. 
  \end{thmNoCTHyp}

  We prove an analogous theorem when the normal subgroup is also $\CAT(0)$ with isolated flats. 
  
  \begin{thmNoCTIFP}
   Let $H \triangleleft G$ be $\CAT(0)$ groups with isolated flats and suppose that $H$ is an infinite, infinite-index normal subgroup of $G$. Then the Cannon--Thurston map does not exist for $(H,G)$.
  \end{thmNoCTIFP}
  
  We obtain a structure theorem for the normal subgroups that arise in the theorems above. If $H$ is a torsion-free infinite-index hyperbolic normal subgroup of a hyperbolic group~$G$, then combined work of Mosher~\cite{mosher}, Paulin~\cite{paulin}, Rips--Sela~\cite{ripssela}, and Bestvina--Feighn~\cite{bestvinafeighn95} proves that $H$ is a free product of finitely many finitely generated free groups and surface groups. See \cite{mitralaminations} for a brief explanation. We prove a similar, but necessarily weaker (see Example~\ref{example:normal_sub}), decomposition theorem in the setting of $\CAT(0)$ groups with isolated flats. We state the result below in the simpler setting that the normal subgroup $H$ is torsion-free; for the full statement, see Section~\ref{sec:subgroup_structure}. 
  
  \begin{thmNormalStructure}
   Let $G$ be a $\CAT(0)$ group with isolated flats, and let $H \triangleleft G$ be an infinite-index torsion-free normal subgroup that is either hyperbolic or $\CAT(0)$ with isolated flats. Then $H \cong F_r * H_1 \ldots * H_n$, where $r,n \geq 0$, $F_r$ denotes the free group of rank~$r$, and $H_i$ is either a closed surface group, a free abelian group, or a $1$-ended group that is either hyperbolic or $\CAT(0)$ with isolated flats and whose JSJ decomposition over elementary subgroups does not contain any rigid vertex groups. 
  \end{thmNormalStructure}
  
   Cannon--Thurston maps are known to exist for certain group pairs in the theorems above with respect to another well-studied boundary. Hruska--Kleiner~\cite{hruskakleiner} proved if $G$ is a $\CAT(0)$ group with isolated flats and $\cP$ is the collection of maximal virtually abelian subgroups of $G$, then $(G, \cP)$ forms a relatively hyperbolic group pair. Bowditch~\cite{bowditch12} showed the pair $(G, \cP)$ has a well-defined boundary, now called the Bowditch boundary  and denoted $\p(G, \cP)$. The $\CAT(0)$ visual boundary contains spheres which arise as the visual boundary of flats, while these spheres are collapsed to points in the Bowditch boundary; see work of Tran~\cite{tran}.  Pal~\cite[Theorem 3.11]{pal} proved the existence of the Cannon--Thurston map for $H \triangleleft G$ with respect to the Bowditch boundaries in the case that the group $H$ is hyperbolic relative to a nontrivial subgroup $H_1$, the group $G$ is hyperbolic relative to $N_G(H_1)$ and weakly hyperbolic relative to $H_1$, and $G$ preserves cusps; see also \cite{bowditch07,mj14}. Under these assumptions, the quotient group is also known to be hyperbolic by work of Mj--Sardar~\cite{mjsardar}. We remark that the natural peripheral structure on a $\CAT(0)$ group with isolated flats described above may contain more than one subgroup.  
   
   \vskip.05in
  
  Two motivating questions regarding Cannon--Thurston maps are 
  \begin{enumerate}
   \item For which groups $H \leq G$ and for which boundaries does the Cannon--Thurston map exist? 
   \item If the Cannon--Thurston map exists, what is its structure? Can this structure be applied to better understand groups and their boundaries?
  \end{enumerate}

  The focus of this paper is the first question above, though we were inspired by results regarding the second. Namely, a prototypical example of a $\CAT(0)$ group with isolated flats is the fundamental group of a mapping torus of a pseudo-Anosov homeomorphism of a surface with negative Euler characteristic and one boundary component. These groups are examples of free-by-cyclic groups $F_n \rtimes \Z$ and act freely and cocompactly on truncated hyperbolic $3$-space~\cite{thurston,otal}. For hyperbolic extensions of free groups, there is a powerful theory of Cannon--Thurston maps that continues to be developed and applied~\cite{mitra,mitralaminations,kapovichlustig,dowdallkapovichtaylor,algomkfirhilionstark}. Thus, we aimed to understand whether these results generalize to the $\CAT(0)$ setting. Our main theorems suggest that generalizations must be studied with respect to the Bowditch boundary.

  Background on Cannon--Thurston maps is given in the survey of Mj~\cite{mj-survey}. The first example of a hyperbolic group pair for which the Cannon--Thurston map does not exist was given by Baker--Riley~\cite{bakerriley}; see also \cite{matsudashinichi,mousley}. 
  
      \begin{remark}
     Every finitely presented group has a {\it Morse boundary} that is well-defined up to homeomorphism, as shown by Cordes~\cite{cordes}, who generalized the {\it contracting boundary} construction of Charney--Sultan~\cite{charneysultan}. The Morse boundary captures hyperbolic-like behavior of a group and coincides with the visual boundary if the group is hyperbolic and the contracting boundary if the group is $\CAT(0)$. For background, see \cite{cordes_survey}. Suppose $H \triangleleft G$, where $H$ is an infinite hyperbolic group and $G$ is a $\CAT(0)$ group with isolated flats. We prove that the Cannon--Thurston map does not exist for the pair $(H,G)$ with respect to their Morse boundaries; see Proposition~\ref{prop:Morse}.  
    \end{remark}
    
  \subsection*{Method of proof}
  
    If $H \triangleleft G$ is an infinite hyperbolic normal subgroup of a $\CAT(0)$ group with isolated flats, we find a dynamical obstruction to the existence of the Cannon--Thurston map. To illustrate our methods, we begin by presenting a special case. Let $G = \la a,b \ra \rtimes_{\phi} \la t \ra$, where $\phi$ is an automorphism of the rank-two free group $\la a,b \ra$ given by $\phi(a) = aba$ and $\phi(b) = ba$. Then $G = \la a,b,t \, | \, t^{-1}at = aba, t^{-1}bt = ba \ra$ is the fundamental group of a mapping torus of a pseudo-Anosov homeomorphism of a surface of genus one with one boundary component. The group $G$ is $\CAT(0)$ with isolated flats and $H = \la a,b \ra$ is a hyperbolic normal subgroup of $G$. The commutator element $[a,b] \leq H$ commutes with the element $t \in G$. By the Flat Torus Theorem~\cite[Theorem II.7.1]{bridsonhaefliger}, if $G$ acts properly and cocompactly by isometries (i.e., {\it geometrically}) on a $\CAT(0)$ space $X$, then the element $[a,b]$ stabilizes a flat $F \subset X$ on which it acts by translation. Hence, the element $[a,b]$ acts trivially on its boundary $\p F \subset \p X$. Since $[a,b] \in H$ has infinite order, the element acts by North-South dynamics on $\p H$ (see Definition~\ref{def:NS}). This difference in dynamics proves a (continuous $H$-equivariant) Cannon--Thurston map cannot exist; see Lemma~\ref{lemma:fp_to_fp}. 
    
    More generally, we prove in Proposition~\ref{prop_elt_H_fstab} that if $H \triangleleft G$ is an infinite hyperbolic normal subgroup of a $\CAT(0)$ group with isolated flats, then there exists an infinite-order element of $H$ which acts by translation on a flat in any $\CAT(0)$ space on which the group $G$ acts geometrically. To prove Proposition~\ref{prop_elt_H_fstab}, we analyze the outer automorphism group of the normal hyperbolic subgroup and apply the following theorem.
    
    \begin{thmPurelyAtoroidal}
    Let $H$ be a hyperbolic group. Then there does not exist a purely atoroidal subgroup of $\Out(H)$ isomorphic to $\Z^2$. 
    \end{thmPurelyAtoroidal}
    
    The proof of Theorem~\ref{thm:pur_ator} relies on two graph of groups decompositions associated to a hyperbolic group $G$, a {\it Stallings--Dunwoody decomposition} of $G$ as a finite graph of groups with finite edge groups and vertex groups that have at most one end and the {\it JSJ decomposition} of the non-Fuchsian one-ended factors in the Stallings--Dunwoody decomposition. As these graph of groups decompositions have aspects invariant under outer automorphisms, the proof reduces to the case that the hyperbolic group is virtually a closed surface group or virtually a free group. The surface group case follows from work of McCarthy~\cite{mccarthy}. The main tool we use in the free group case is the work of Feighn--Handel~\cite{fh09}.
        
   If $H \triangleleft G$ are both $\CAT(0)$ groups with isolated flats, and $H$ is an infinite-index subgroup, then it is possible that an element of $H$ acts by North-South dynamics on $\p H$, and acts by translation on a flat in a $\CAT(0)$ model space for $G$. In this case, our previous dynamical arguments (see Lemma~\ref{lemma:crit_nonexist}) prove the Cannon--Thurston map does not exist. However, there are examples in which this is not the case; see Example~\ref{example:normal_sub} and Figure~\ref{figure:noCT}. In these situations, we employ a different strategy of proof: we show there are sequences of elements in the subgroup $H$ so that the corresponding orbit points converge to the same point in $\p H$, but converge to different points in the boundary of a flat in $\p G$. For example, with $H \triangleleft G$ as in the beginning of this subsection, the sequences $[a,b]^n$ and $[a,b]^n\phi^n(a) = [a,b]^nt^nat^{-n}$ converge to the same point in $\p H$. However, these sequences converge to different points in $\p G = \p X$ as the sequence $[a,b]^nt^nat^{-n}$ approximates the diagonal in the flat in $X$ stabilized by $[a,b]$ and $t$. The theory of closest-point projections for relatively hyperbolic group pairs of Dru\c{t}u--Sapir~\cite{drutusapir} and Sisto~\cite{sisto} constitutes an important tool in our analysis.  

   \vskip.1in
   
 \subsection*{Acknowledgements} The authors are thankful for helpful discussions with Fran\c{c}ois Dahmani, Chris Hruska, Mahan Mj, Kim Ruane, Michah Sageev, and Daniel Woodhouse. The second and fourth authors were supported by the Israel Science Foundation (Grant 1026/15). The second and fifth authors were partially supported at the Technion by a Zuckerman STEM Leadership Fellowship.  The second author is also partially supported by an ETHZ Postdoctoral Fellowship. The third author was partially supported by the Israel Science Foundation (Grant 662/15) and by the Deutsche Forschungsgemeinschaft (DFG, German Research Foundation) under Germany's Excellence Strategy -- EXC 2044 -- 390685587, Mathematics Münster: Dynamics -- Geometry -- Structure, and acknowledges the hospitality of the Hausdorff Research Institute for Mathematics. The fifth author was partially supported by the Azrieli Foundation.

   \vskip.2in
 
  \section{Preliminaries}
 
  \subsection{Visual boundary}
  
  \begin{defn}
   Let $X$ be a proper geodesic metric space which is hyperbolic or $\CAT(0)$. The {\it visual boundary of $X$}, denoted $\p X$, is the set of equivalence classes of geodesic rays in $X$, where two rays $c_1,c_2:[0,\infty) \rightarrow X$ are {\it equivalent} if there exists a constant $D$ so that $d(c_1(t),c_2(t)) \leq D$ for all $t \in [0,\infty)$. 
  \end{defn}

  There is a natural topology on $X \cup \p X$ called the {\it cone topology}. With respect to this topology, both $\p X$ and $X \cup \p X$ are compact Hausdorff spaces whenever $X$ is proper. An isometry of $X$ induces a homeomorphism of $\p X$. See \cite{bridsonhaefliger,kapovichbenakli} for details. For the remainder of the paper, we assume that the hyperbolic and $\CAT(0)$ spaces we consider are proper geodesic metric spaces. 
      
  \begin{defn}[North-South dynamics] \label{def:NS}
      If $Z$ is a topological space, then $g \in \Homeo(Z)$ acts with {\it North-South dynamics} if $g$ fixes two distinct points $g^{+ \infty},g^{- \infty} \in Z$ and for any open sets $U,V \subset Z$ with $g^{+ \infty} \in U$ and $g^{- \infty} \in V$ there exists $n \in \N$ so that $g^n(Z - V) \subset U$ and $g^{-n}(Z-U) \subset V$. 
  \end{defn}
  
  \begin{thm}  \cite{coorneartdelzantpapadopoulos,alonsoetal,ghysdelaharpe,gromov} \label{thm_hyp_NS}
   If $H$ is a hyperbolic group and $a \in H$ is an element of infinite order, then $a$ acts by North-South dynamics on $\p H$. 
  \end{thm}
  
  \begin{defn}
   Let $X$ be a $\CAT(0)$ space, and let $h \in \Isom(X)$ be a loxodromic element. The element $h$ is called a {\it rank-one isometry of $X$} if there exists an axis for $h$ which does not bound an isometrically embedded copy of a Euclidean half-plane.  
  \end{defn}
  
   \begin{thm}  \label{lemma_rk1_NS} \cite[Lemma 4.4]{hamenstadt}
    Let $X$ be a proper $\CAT(0)$ space, and let $G \leq \Isom(X)$ be a subgroup whose limit set consists of at least two points and which does not fix a point in $\p X$. If $g \in G$ has rank one, then $g$ acts on $\p X$ with North-South dynamics.      
  \end{thm}
  
  A lemma similar to the following can be found in \cite[Lemma 2.10]{jeonkapovichleiningerohshika}. 
  
  \begin{lemma} \label{lemma:fp_to_fp}
    Let $H \leq G$. Suppose $H$ acts geometrically on $Y$ and $G$ acts geometrically on $X$, where $X$ and $Y$ are hyperbolic or $\CAT(0)$ spaces. If $a \in H$ acts by North-South dynamics on $\p Y$ fixing the points $a^{\pm \infty} \in \p Y$ and $\hi:\p Y \rightarrow \p X$ is continuous and $H$-equivariant, then $a$ acts by North-South dynamics on $\hi(\p Y)$ fixing the points $\hi(a^{\pm \infty}) \in \p X$.
  \end{lemma}
  \begin{proof}
   Since $\hi$ is $H$-equivariant, the element $a$ fixes $\hi(a^{+\infty})$ and $\hi(a^{-\infty})$. Since $a$ has infinite order, $\hi(a^{+\infty}) \neq \hi(a^{-\infty})$. Let $U,V \subset \hi(\p Y)$ be open sets with $\hi(a^{+\infty}) \subset U$ and $\hi(a^{-\infty}) \subset V$. Since $\hi$ is continuous, the sets $\hi^{-1}(U)$ and $\hi^{-1}(V)$ are open in $\p Y$ and contain the points $a^{+\infty}$ and $a^{-\infty}$, respectively. Since $a$ acts by North-South dynamics on $\p Y$, there exists $n \in \N$ so that $a^n(\p Y - \hi^{-1}(V)) \subset \hi^{-1}(U)$. Since $\hi$ is $H$-equivariant, $a^n(\hi(\p Y) - V) \subset U$. Thus, the element~$a$ acts by North-South dynamics on $\hi(\p Y)$ fixing $\hi(a^{\pm \infty})$. 
  \end{proof}

  \subsection{\texorpdfstring{$\CAT(0)$}{CAT(0)} groups with isolated flats}
  
  \begin{defn}
   Let $X$ be a $\CAT(0)$ space. For $k \geq 2$, a $k$-flat in $X$ is an isometrically embedded copy of Euclidean space $\E^k$. 
  \end{defn}

  \begin{defn} \label{def:cat0IFP}
   Let $X$ be a $\CAT(0)$ space, and suppose a group $G$ acts geometrically on $X$. The space $X$ has the {\it isolated flats property} if $X$ is not quasi-isometric to $\E^n$ and there is a non-empty $G$-invariant collection of flats $\cF$, of dimension at least two, such that the following conditions hold:
   \begin{enumerate}
    \item There exists a constant $D < \infty$ such that each flat in $X$ lies in the $D$-tubular neighborhood of some $F \in \cF$.
    \item For every $\rho < \infty$ there exists $\kappa(\rho)<\infty$ such that for any two distinct flats $F,F' \in \cF$, $\diam(N_\rho(F) \cap N_\rho(F')) <\kappa(\rho)$. 
   \end{enumerate}
   If a group $G$ is {\it $\CAT(0)$ with isolated flats} if $G$ acts geometrically on a $\CAT(0)$ space with the isolated flats property. 
  \end{defn}

  \begin{remark}
   By definition, a $\CAT(0)$ space with isolated flats contains a non-trivial flat of dimension at least two. By \cite[Lemma 3.1.2]{hruskakleiner}, a $\CAT(0)$ group $G$ with isolated flats has a non-trivial free abelian subgroup of rank at least two. Thus, if $H \leq G$ is hyperbolic, then $H$ has infinite-index in $G$. 
  \end{remark}

    The following theorem collects the results of work of Hruska--Kleiner~\cite{hruskakleiner} relevant to this paper. We refer the reader to \cite{hruskakleiner} for additional background.
    
    \begin{thm} \cite[Theorems 1.2.1, 1.2.2]{hruskakleiner} \label{thm:hruska_kleiner}
      Let $X$ be a $\CAT(0)$ space, and let $G$ be a group acting geometrically on $X$. 
     \begin{enumerate}
      \item The following are equivalent.
	  \begin{enumerate}
	   \item The space $X$ has the isolated flats property.
	   \item The space $X$ is relatively hyperbolic with respect to a family of flats $\cF$.
	   \item The group $G$ is a relatively hyperbolic group with respect to a collection of virtually abelian subgroups of rank at least two. 
	  \end{enumerate}
      \item If the space $X$ has the isolated flats property, then the visual boundary $\p X$ is a group invariant of $G$.
     \end{enumerate}
    \end{thm}
    
    \begin{lemma} \label{cor:char_isoms}
     Let $G$ act geometrically on a $\CAT(0)$ space $X$ with isolated flats. If $g \in G$ has infinite order, then either $g$ acts by translation on a flat in $X$ or $g$ has rank-one. 
    \end{lemma}
    \begin{proof}
     Suppose the element $g$ does not have rank one. Then an axis for $g$ bounds a half-flat $F_0$ in~$X$.       
     By Theorem~\ref{thm:hruska_kleiner}, the space $X$ is hyperbolic relative to a family of flats, so this half-flat $F_0$ must be contained in a bounded neighborhood of a flat $F$ in $X$. Let $\ell$ be an axis for~$g$. The endpoints $\alpha^+, \alpha^-$ of the line $\ell$ lie in the boundary of the half-flat $F_0$ in $\p X$, and hence lie in the boundary of the flat $F$ in $\p X$. By work of Tran~\cite[Main Theorem]{tran}, there is a $G$-equivariant continuous map  $\pi: \p X\rightarrow \p(G, \cP)$, where $\p(G, \cP)$ denotes the Bowditch boundary for $G$ defined with respect to the peripheral structure given by Hruska--Kleiner in Theorem~\ref{thm:hruska_kleiner}. Moreover, there exists a parabolic fixed point $z \in \p(G,\cP)$ so that $\pi(\p F) = z$. Since $g$ fixes $\alpha^+$ and $\alpha^-$, the element $g$ fixes $\pi(\alpha^{\pm}) = z \in \p (G, \cP)$.  Thus, by work of Bowditch~\cite{bowditch12} (see also~\cite[Section 3]{hruska10}), the element $g$ is contained in a conjugate of a peripheral subgroup of $G$. Thus, $g$ is contained in a virtually abelian subgroup of $G$ of rank at least two; hence, $g$ acts by translation on a flat in $X$ by the Flat Torus Theorem~\cite[Corollary II.7.]{bridsonhaefliger}.  
    \end{proof}
          
    
   \begin{lemma} \label{lemma:prod_cat0}
    Let $G$ be $\CAT(0)$ with isolated flats. If $A \times B \leq G$ and $A$ and $B$ are infinite, then $A$ and $B$ are virtually abelian.
   \end{lemma}
   \begin{proof}
    By the Flat Torus Theorem~\cite[Theorem II.7.1]{bridsonhaefliger}, every infinite-order element of $A$ has an axis that spans a flat with an axis of every infinite-order element of the subgroup $B$. Thus if $a,a' \in A$ and $b \in B$ are infinite, then an axis for $b$ is contained in a flat $F$ stabilized by $\la a,b \ra$ and an axis for $b$ is contained in a flat $F'$ stabilized by $\la a',b \ra$. As in the proof of Lemma~\ref{cor:char_isoms}, there is a continuous $G$-equivariant map $\pi: \p X \rightarrow \p (G, \cP)$. Since an axis for $b$ is contained in the flats $F$ and $F'$, there exists a parabolic point $z \in \p(G, \cP)$ so that $\pi(F) = \pi(F') = z$ and the elements $a$, $a'$, and $b$ stabilize $z$ by \cite[Main Theorem]{tran}. Thus, by work of Bowditch~\cite{bowditch12}, the elements $a$ and $a'$ belong to the same conjugate of a peripheral subgroup of $G$. Thus, $A$ is virtually abelian. Similarly, $B$ is virtually abelian. 
   \end{proof}

 \subsection{Stallings--Dunwoody and JSJ decompositions} \label{sec:grushko_jsj}
  
    To describe the structure of normal subgroups, we apply the Stallings--Dunwoody and JSJ decomposition theories. We recall the relevant statements below. For background on graphs of groups, see \cite{scottwall,serre}. The graph of groups decomposition given in the next theorem is called a {\it Stallings--Dunwoody decomposition} of $G$.
  
  \begin{thm} \cite{dunwoody, stallings} \label{thm:dunwoody}
    If $G$ is a finitely presented group, then $G$ splits as a finite graph of groups with finite edge groups and vertex groups that have at most one end. 
  \end{thm}
  
  \begin{remark}
   \label{remark_one_ended}
   If the group $G$ has one end, then the graph of groups decomposition given by the above theorem is trivial.
   This is the case in particular when $G$ has a finitely generated, infinite, infinite-index normal subgroup (as is the set up of Theorems~\ref{thm_noCT_hyp} and~\ref{thm_noCT_IFP}), since an infinite-ended group cannot have such a subgroup.
   This was proved by Karrass--Solitar; it is a special case of \cite[Theorem 10]{karrass_solitar_amalgam} when $G$ is an amalgam and of \cite[Theorem 9]{karrass_solitar_hnn} when $G$ is an HNN extension.
   (Alternatively, appealing to more modern technology, infinite-ended groups have positive first $L^2$-Betti number, which is an obstruction to having such a subgroup.)
  \end{remark}
  
  \begin{lemma} \label{lemma_one_end}
   Let $H_v$ be a one-ended vertex group in a Stallings--Dunwoody decomposition of a finitely presented group $H$, and let $\phi \in \Out(H)$. Then there exists $m \in \N^+$ so that $\phi^m([H_v]) = [H_v]$. 
  \end{lemma}
  \begin{proof}
  Let $T$ and $T'$ be the Bass--Serre trees of two Stallings--Dunwoody decompositions of $H$. We claim that $T$ and $T'$ have the same one-ended vertex stabilizers up to conjugation. 
  
  To prove the claim, let $H_1, \ldots, H_n \leq H$ and $H_1', \ldots H_m' \leq H$ be representatives of one-ended vertex stabilizers of $T/H$ and $T'/H$, respectively. Since $H_i$ is one-ended, the subgroup $H_i$ fixes a unique vertex of the tree $T'$. Thus $H_i \leq gH_j'g^{-1}$ for some $g \in H$ and $1 \leq j \leq m$. Similarly, the group $gH_j'g^{-1}$ is one-ended, so $gH_j'g^{-1} \leq h H_k h^{-1}$ for some $h \in H$ and $1 \leq k \leq n$. Thus, $H_i \leq gH_jg^{-1} \leq hH_kh^{-1}$. The subgroups $H_i$ and $hH_kh^{-1}$ fix unique vertices in the tree $T$, so $H_i=H_k$. Therefore, for every $1\leq i \leq n$, there exists $j \in \{1, \ldots, m\}$ such that $[H_i] = [H_j']$ and vice versa, concluding the proof of the claim. 
  
  Let $H_v$ be a one-ended vertex stabilizer of the tree $T$ and let $T' = T\cdot \phi$. Since $T$ and $T'$ have the same one-ended vertex stabilizers, there exists an $m \in \N^+$ such that $\phi^m([H_v]) = [H_v]$.  
  \end{proof}
  
  \begin{remark}
   Lemma~\ref{lemma_one_end} is a specific instance of a general fact mentioned by Guirardel--Levitt~\cite[Introduction]{guirardellevitt17}: Two $G$-trees belonging to the same deformation space over a class of subgroups $\cA$ have the same vertex stabilizers, provided one restricts to the vertex groups not in $\cA$. 
  \end{remark}

    The graph of groups decomposition for a hyperbolic group or a $\CAT(0)$ group with isolated flats provided by the next two theorems is called the {\it JSJ decomposition} of $G$ over two-ended or elementary subgroups, respectively. For $1$-ended hyperbolic groups that are not Fuchsian, we use the language due to Bowditch~\cite{bowditch}; see \cite{guirardellevitt17} for background on JSJ decompositions. We note that there are various uses of the term ``elementary'' in the literature. To avoid confusion, we reserve the term in this paper for the JSJ theory (see Definition~\ref{def:elementary}), and write the assumptions explicitly in other cases.
  
  \begin{defn}
     A {\it Fuchsian group} is a finitely generated group that is not finite or two-ended and acts properly discontinuously on the hyperbolic plane. A {\it cocompact Fuchsian group} is a Fuchsian group that acts cocompactly on the hyperbolic plane. In particular, a cocompact Fuchsian group is virtually a {\it closed surface group}, the fundamental group of a closed surface. 
  \end{defn}
  
  \begin{defn} \cite{bowditch}  
    A {\it bounded Fuchsian group} is a Fuchsian group that is convex cocompact (i.e., it acts cocompactly on the convex hull of its limit set in $\Hy^2$) but is not virtually a closed surface group. The convex core of the quotient is a compact orbifold with non-empty boundary consisting of a disjoint union of compact $1$-orbifolds. The {\it peripheral subgroups} are the maximal two-ended subgroups which project to the fundamental groups of the boundary $1$-orbifolds. A {\it hanging Fuchsian} subgroup~$H$ of a group is a virtually-free quasi-convex subgroup together with a collection of {\it peripheral} two-ended subgroups, which arise from an isomorphism of $H$ with a bounded Fuchsian group. A {\it full quasi-convex subgroup} of a group $G$ is a quasi-convex subgroup that is not a finite-index subgroup of any strictly larger subgroup of $G$. 
  \end{defn}
  
   \begin{thm}\cite[Thm 0.1]{bowditch} \label{thm:bow_jsj}
    Let $G$ be a one-ended hyperbolic group that is not Fuchsian. There is a canonical \emph{JSJ decomposition} of $G$ as the fundamental group of a graph of groups such that each edge group is 2-ended and each vertex group is either (1)  $2$-ended; (2) maximal hanging Fuchsian; or, (3) a maximal quasi-convex subgroup not of type (2). These types are mutually exclusive, and no two vertices of the same type are adjacent. Every vertex group is a full quasi-convex subgroup. Moreover, the edge groups that connect to any given vertex group of type (2) are precisely the peripheral subgroups of that group.      
  \end{thm} 

  \begin{remark}
   The JSJ decomposition due to Bowditch is constructed via the topology of the visual boundary of the hyperbolic group $G$. It follows from construction that any automorphism of $G$ induces an isomorphism of the Bass--Serre tree of the JSJ decomposition.
  \end{remark}
  
  For many relatively hyperbolic groups $(G, \cP)$, there is also a canonical JSJ decomposition due to Guirardel--Levitt~\cite{guirardellevitt11} which is invariant under the subgroup of $\Out(G)$ that preserves the set of conjugacy classes of the peripheral subgroups of $G$. Denote this subgroup by $\Out(G, \cP)$. We state their results in the special case of a $\CAT(0)$ group with isolated flats. In this special case, $\Out(G, \cP) = \Out(G)$ because the set of peripheral subgroups is the set of maximal abelian subgroups of $G$ which is preserved by every automorphism.  
  
  \begin{defn} \label{def:elementary}
   Let $G$ be hyperbolic relative to a finite collection $\cP = \{P_1, \ldots, P_k\}$. A subgroup of $G$ is {\it elementary} if it is either contained in a conjugate of some $P_i \in \cP$ or if it is infinite, virtually cyclic, and not contained in any $P_i \in \cP$.
  \end{defn}
  
   \begin{thm} \label{thm:jsj_gl} \cite[Theorem 4]{guirardellevitt11}\cite[Section 3]{guirardellevitt15}
    Let $G$ be hyperbolic relative to $\cP$, a collection of maximal virtually abelian subgroups of rank at least two. Also suppose $G$ is one-ended. Then there is an elementary JSJ tree for $G$ relative to $\cP$ which is invariant under $\Out(G)$. The vertex groups in the JSJ decomposition have one of four types: (0.a) rigid; (0.b) maximal hanging Fuchsian; (1.a) maximal parabolic: conjugate to a $P_i \in \cP$; and (1.b) maximal loxodromic: maximal virtually cyclic subgroup of $G$ and not contained in a $P_i \in \cP$. Vertex groups of type (0.*) are only adjacent to vertex groups of type (1.*).
    \end{thm}
  
  \begin{notation}
    Let $G$ be a one-ended group that is either hyperbolic or $\CAT(0)$ with isolated flats, and let $T$ be the JSJ tree for $G$ given by either Theorem~\ref{thm:bow_jsj} or Theorem~\ref{thm:jsj_gl}, respectively. Let $\Out^0(G) \leq \Out(G)$ denote the outer automorphisms of $G$ that act trivially on the graph $T/G$. The subgroup $\Out^0(G)$ has finite index in the group $\Out(G)$.  
  \end{notation}
    
    A key feature of the JSJ decompositions considered in this paper is that outer automorphisms in $\Out^0(G)$ ``restricted to'' rigid vertex groups have finite-order. This fact is made more precise as follows. 
 
 \begin{remark} \cite[Section 2.6]{guirardellevitt15} Let $G$ be a one-ended group that is either hyperbolic or $\CAT(0)$ with isolated flats, and let $T$ be the JSJ tree for $G$. For a vertex $v$ of $T/G$, there is a map $\rho_v \colon \Out^0(G) \to \Out(G_v)$ given as follows: choose a representative $\tilde{v}$ of $v$ in $T$ that is stabilized by $G_v$. For $\phi \in \Out^0(G)$, choose a representative $\Phi$ of $\phi$ such that the induced map on $T$ fixes $\tilde{v}$. Then $\rho_v(\phi)$ is represented by the restriction of $\Phi$ to $G_v$. We say $\rho_v(\phi)$ is the {\it restriction} of $\phi$ to $G_v$.  \end{remark}
  \begin{thm} \cite{paulin,bestvinafeighn95} \label{thm:hyp_rig_fin}
    Let $G$ be hyperbolic, and let $G_v \leq G$ be a vertex group of type~(3) in Theorem~\ref{thm:bow_jsj}. Suppose $\phi \in \Out^0(G)$, and let $\phi'$ be the restriction of $\phi$ to $G_v$. Then $\phi'$ has finite order. 
  \end{thm}
  
  \begin{thm} \cite[Proposition 4.1]{guirardellevitt15} \label{thm:cat_rig_fin}
   Let $G$ be $\CAT(0)$ with isolated flats, and let $G_v \leq G$ be a vertex group of type (0.a) in Theorem~\ref{thm:jsj_gl}. Suppose $\phi \in \Out^0(G)$, and let $\phi'$ be the restriction of $\phi$ to $G_v$. Then $\phi'$ has finite order. 
  \end{thm}

  \subsection{Limit set}
  
  \begin{defn}
   Let a group $G$ act by isometries on a hyperbolic or $\CAT(0)$ space $X$, and let $x \in X$. The {\it limit set of $G$ defined at $x$}, denoted $\Lam_x G$, is the set of accumulation points in $\p X$ of the $G$-orbit of $x$ in the space $X$.
  \end{defn}

  The limit set does not depend on the choice of basepoint.
  
  \begin{lemma}
   If $G$ acts properly discontinuously and cocompactly by isometries on a proper hyperbolic or $\CAT(0)$ space~$X$, and $x,y \in X$, then $\Lam_xG = \Lam_yG$.
   \qed
  \end{lemma}
  
  
  \begin{lemma} \cite[Proof of Lemma 5.2]{hamenstadt} \label{lemma_equiv_defs}
   Let $X$ be a proper $\CAT(0)$ space, and let $G \leq \Isom(X)$ be a subgroup whose limit set contains more than two points and so that $G$ does not fix a point on~$\p X$. Suppose $G$ contains a rank-one isometry.  Let $x \in X$. The limit set $\Lam_xG$ is the smallest non-empty closed $G$-invariant set in $\p X$. \qed
  \end{lemma}

  \begin{lemma} \label{lemma_rank_one_exists} 
   Let $G$ act geometrically on a $\CAT(0)$ space $X$ with isolated flats. If $H \triangleleft G$ is infinite, then $H \leq \Isom(X)$ contains a rank-one isometry of $X$. 
  \end{lemma}
  \begin{proof}
    The subgroup $H$ contains an infinite-order element $h$. Indeed, the group $G$ satisfies the Strong Tits Alternative: every subgroup of $G$ is either finitely generated virtually abelian or contains a free subgroup of rank two \cite[Theorem 1.2.2]{hruskakleiner}. By Theorem~\ref{thm:hruska_kleiner}, the group $(G, \cP)$ forms a relatively hyperbolic group pair, where $\cP$ is a collection of virtually abelian subgroups of $G$ of rank at least two. The peripheral subgroups of $G$ are almost malnormal \cite[Section 3.3]{farb98}, so $H \not\subset P_i$ for any $P_i \in \cP$. Thus, either $H$ is two-ended or contains a free subgroup of rank two. So, $H$ contains an infinite-order element $h$.    
    
    By Lemma~\ref{cor:char_isoms}, we may assume the element $h$ acts by translation on a flat in $X$. So, $h \in P_i^{g_0}$, a conjugate of a peripheral subgroup $P_i$ of $G$. Consider $g \in G \backslash P_i^{g_0}$. Then $ghg^{-1} \in P_i^{gg_0}$, a different conjugate of a peripheral subgroup of $G$. Since $H \triangleleft G$, the element $ghg^{-1} = h' \in H$. We can ensure that the product $hh'$ does not stabilize a flat in $X$. Indeed, the elements $h$ and $h'$ act elliptically on the JSJ tree for $G$ given by Theorem~\ref{thm:jsj_gl} (since $G$ is one-ended, see Remark~\ref{remark_one_ended}), and we can choose the conjugating element $g$ to translate the fixed points of $h$ far enough that $h$ and $h'$ fix disjoint sets of vertices. So, their product acts as a loxodromic on the tree, and is thus not contained in a peripheral subgroup. By Lemma~\ref{cor:char_isoms}, the element $hh'$ has rank one. 
  \end{proof}

  \begin{lemma} \label{lemma_cat_ls_normal}
    Let $G$ act geometrically on a $\CAT(0)$ space~$X$ with isolated flats. If $H \triangleleft G$ is a normal subgroup, then $\Lam_xH = \Lam_xG$ for all $x \in X$. 
  \end{lemma}
  \begin{proof}
    Suppose a group $G$ acts geometrically on a $\CAT(0)$ space~$X$ with isolated flats, and let $x \in X$. Let $H \triangleleft G$ be a normal subgroup. Since $\Stab(g\cdot\Lam_xH) = g \Stab(\Lam_xH) g^{-1} \supseteq gHg^{-1} = H$ for all $g \in G$, the set $g \cdot\Lam_xH$ is $H$-invariant for all $g \in G$. 
    
    We first show that $g \cdot \Lam_xH = \Lam_x H$ for all $g \in G$. The subgroup $H$ contains a rank-one isometry $h \in H$ by Lemma~\ref{lemma_rank_one_exists}. Let $h^{\infty}  = \lim_{n \rightarrow \infty} h^n\cdot x \in \Lambda_xH$. Then the sequence $\{h^n \cdot gh^{\infty}\}_{n \in \N}$ converges to $h^{\infty}$ by North-South dynamics, unless $g \cdot h^\infty = h^{-\infty}$. In the general case since $g \cdot \Lam_xH$ is closed and $H$-invariant we have the limit $h^\infty \in g \cdot \Lam_xH$, so in either case $g \cdot \Lam_xH \cap \Lam_x H \neq \emptyset$. By Lemma~\ref{lemma_equiv_defs}, the set $\Lambda_xH$ is the smallest closed $H$-invariant set in $\p X$. Since $g \cdot \Lam_xH \cap \Lam_x H \neq \emptyset$ and $g \cdot \Lam_xH$ is closed and $H$-invariant, $g \cdot \Lam_xH = \Lam_xH$. Therefore, the set $\Lam_xH \subset \Lam_xG$ is closed and $G$-invariant. Thus, Lemma~\ref{lemma_equiv_defs} implies $\Lam_xH = \Lam_xG$.
  \end{proof}

 \subsection{Cannon--Thurston map} \label{sec:CT_map}
  Morally, the Cannon--Thurston map for $H \hookrightarrow G$ should be an extension of the subgroup inclusion to the boundary, as it is in the hyperbolic case.
  To define the boundary of a $\CAT(0)$ group $G$ we however need a $\CAT(0)$ model space $X$ on which it acts.
  Rather than develop a theory of the compactification $G \cup \p X$, which would enable a definition of Cannon--Thurston maps simply `extending the group inclusion to the boundary', we exploit the convenience afforded by the identification of the limit set $\Lam_x G$ with the boundary $\p X$ and make the following definition.
  \begin{defn}[Cannon--Thurston map] \label{def:CT_map}
      Let $H \leq G$ such that $H$ acts geometrically on $Y$ and $G$ acts geometrically on $X$, where $X$ and $Y$ are hyperbolic or $\CAT(0)$ spaces. Let $x \in X$ and $y \in Y$. A \emph{Cannon--Thurston map} for the pair $(H \curvearrowright Y, G \curvearrowright X)$ is the map $f \colon \p Y \to \p X$ defined by \[f\left(\lim_{i \rightarrow \infty} h_i \cdot y\right)=\lim_{i \rightarrow \infty} h_i \cdot x,\] provided $f$ is well-defined and continuous. Being well-defined means that $f(\xi)$ does not depend on the choice of the sequence of group elements $h_i \in H$ such that $\underset{i \rightarrow \infty}{\lim} h_i \cdot y = \xi$.
    \end{defn}
      
  \begin{remark}
   If $H \leq G$ are hyperbolic or $\CAT(0)$ with isolated flats, then the choice of model spaces $X$ and $Y$ does not affect the existence of the Cannon--Thurston map for the pair $H \leq G$. This follows from the fact that the visual boundary is a quasi-isometry invariant for hyperbolic groups \cite{ gromov} and a group invariant for the $\CAT(0)$ with isolated flats case \cite[Corollary 4.1.9]{hruskakleiner}. Thus, we may define a {\it Cannon--Thurston map} for the pair $(H,G)$. Note that the example of Croke--Kleiner \cite{crokekleiner} shows that without the isolated flats assumption, $\CAT(0)$ groups do not in general have a well-defined boundary. This is critical, since for such an example, with $G$ acting on $X$ and $Y$ with $\p X \not \cong \p Y$, there cannot exist Cannon--Thurston maps for both of the pairs $(G \curvearrowright X, G \curvearrowright Y)$ and $(G \curvearrowright Y, G \curvearrowright X)$ because the compositions would have to be the identity maps on $\p X$ and $\p Y$ implying $\p X \cong \p Y$.
  \end{remark}
      
  \begin{remark} \label{rem:H_equivariant}
   By definition, the Cannon--Thurston map is $H$-equivariant.
  \end{remark}
    
  If $H$ and $G$ are hyperbolic groups and $H$ is a quasi-convex subgroup of $G$, then the Cannon--Thurston map exists for the pair $(H,G)$ by \cite{ghysdelaharpe}. Hruska--Kleiner prove an analogous statement in the setting of $\CAT(0)$ groups with isolated flats~\cite[Theorem 4.1.8]{hruskakleiner}.

  \section{The structure of normal subgroups} \label{sec:subgroup_structure}
       
  \begin{notation} \label{nota_out}
   Let $1 \rightarrow H \rightarrow G \xrightarrow{p} Q \rightarrow 1$ be a short exact sequence of groups. This sequence induces a homomorphism $\phi: Q \rightarrow \Out(H)$. If $q \in Q$, let $\phi_q = \phi(q)$. If $A \leq G$ is a subgroup, let $[A]$ denote the conjugacy class of $A$ in $G$. 
  \end{notation}
  
  \begin{lemma} \label{lemma:inf_ord_elt}
   Let $G$ be a $\CAT(0)$ group with isolated flats, and suppose $H \triangleleft G$ is a normal subgroup of infinite-index such that $H$ is either hyperbolic or $\CAT(0)$ with isolated flats. Then the following hold.
   \begin{enumerate}
    \item Let $q \in Q$ be an infinite order element and $H' \leq H$. If $\phi_q \in \Out(H)$ restricts to an outer automorphism $\phi_q' \in \Out(H')$, then the group $G$ contains a subgroup isomorphic to the semi-direct product $H' \rtimes_{\phi_q'} \Z$.
    \item The group $Q = G/H$ contains an element of infinite-order.
    \item If $q \in Q$ has infinite order, then the element $\phi_q \in \Out(H)$ has infinite order. 
   \end{enumerate}
  \end{lemma}
  \begin{proof}
    We first prove (1). By assumption there exists $g \in G$ with $p(g) = q$ such that $H'^g = H'$ inducing $\phi_q' \in \Out(H')$. An elementary argument proves the subgroup $\la H', g \ra \leq G$ is isomorphic to the group $H' \rtimes_{\phi_q'} \Z$.  
    
    To prove (2), suppose first that $H$ is hyperbolic. By \cite[Lemma 3.1.2]{hruskakleiner}, for every flat $F \subset X$, there exist elements $a,b \in \Stab_G(F)$ with $\la a,b \ra \cong \Z^2$. Since $H$ is hyperbolic, the images $p(a)$ and $p(b)$ in the quotient $p:G \rightarrow Q$ cannot both be finite; otherwise, $H$ would contain an abelian subgroup of rank two. 
    
    Suppose now that $H$ is $\CAT(0)$ with isolated flats. Let $G' \leq G$ be a maximal-rank free abelian subgroup of $G$; the rank $n$ of $G'$ is at least two. Suppose $G$ acts geometrically on a $\CAT(0)$ space $X$. The quotient space $X/G$ contains an $n$-torus $T$ with fundamental group $G'$. The infinite-index subgroup $H \leq G$ yields an infinite cover $X/H$ of $X/G$. Consider the full pre-image of $T$ in $X/H$. This pre-image need not be connected and will not be compact. The pre-image of the torus could fail to be compact if either the pre-image has infinitely many compact components, or if the pre-image  contains a component that is not compact. However, the pre-image cannot contain infinitely many compact $n$-torus components since $H$ has finitely many conjugacy classes of rank-$n$ free abelian subgroups by \cite[Lemma 3.1.2]{hruskakleiner}. 
  
    To prove (3), suppose $q \in Q$ is an element of infinite order. There exists $g \in G$ with $p(g) = q$. The subgroup $\la H, g \ra \leq G$ is isomorphic to the group $H \rtimes_{\phi_q} \Z$ by (1). If $\phi_q$ has finite order, then this semi-direct product is virtually a product $H \times \Z$. Since $G$ is $\CAT(0)$ with isolated flats, Lemma~\ref{lemma:prod_cat0} forces $H$ to be virtually abelian, and since $H$ is either hyperbolic or $\CAT(0)$ with isolated flats, it must then in fact be virtually cyclic. However, by Lemma~\ref{lemma_cat_ls_normal}, the group $H$ is not two ended. Thus, the element $\phi_q$ has infinite order. 
  \end{proof} 
  
  The next theorem characterizes the isomorphism types of normal subgroups that are either hyperbolic or $\CAT(0)$ with isolated flats in a $\CAT(0)$ group with isolated flats. 
  
  \begin{thm} \label{thm_structure_normal}
   Let $G$ be a $\CAT(0)$ group with isolated flats, and let $H \triangleleft G$ be an infinite-index normal subgroup that is either hyperbolic or $\CAT(0)$ with isolated flats. Let $\cH$ be a Stallings--Dunwoody decomposition of $H$. If $H_v \leq H$ is a one-ended vertex group of $\cH$, then $H_v$ is either 
   \begin{enumerate}
    \item virtually a closed surface group, or 
    \item virtually abelian, or 
    \item hyperbolic and the JSJ decomposition of $H_v$ over $2$-ended subgroups given by Theorem~\ref{thm:bow_jsj} does not contain vertex groups of type~(3), or 
    \item  a $\CAT(0)$ group with isolated flats whose JSJ decomposition over elementary subgroups given by Theorem~\ref{thm:jsj_gl} does not contain any vertex groups of type~(0.a).
   \end{enumerate}
  \end{thm}
  \begin{proof}
   Suppose $H_v$ is neither virtually a closed surface group, nor virtually abelian. We first show that $H_v$ is either hyperbolic or $\CAT(0)$ with isolated flats. Since $H_v$ is a vertex group of a splitting of $H$ over finite groups, a Cayley graph of $H_v$ quasi-isometrically embeds in a Cayley graph of $H$. Therefore, if $H$ is hyperbolic, then $H_v$ is also hyperbolic. Suppose now that $H$ is $\CAT(0)$ with isolated flats. By \cite[Theorem 1.3]{hruskaruane}, the subgroup $H_v$ is a $\CAT(0)$ group. By Theorem~\ref{thm:hruska_kleiner}, the group $H$ is hyperbolic relative to the collection of its maximal virtually abelian subgroups. Thus, the coned-off Cayley graph of $H$ with respect to these subgroups is hyperbolic \cite{farb98}. By \cite[Theorem 8.29]{drutusapir},  $H_v$ is hyperbolic relative to subgroups of the form $H_v \cap gP_ig^{-1}$, where $g \in G$ and $P_i$ is a peripheral subgroup of $H$. If $H_v$ does not contain a $\Z^2$ subgroup, then we claim that $H_v$ is hyperbolic. Indeed, the subgroup $H_v$ does not intersect any $\Z^2$ subgroup non-trivially. Thus, no vertex of the quasi-isometrically embedded Cayley graph of $H_v$ gets coned-off. Hence, the subgroup $H_v$ is hyperbolic. Otherwise, $H_v$ is hyperbolic relative to virtually abelian subgroups, and hence by Theorem~\ref{thm:hruska_kleiner}, the subgroup $H_v$ is $\CAT(0)$ with isolated flats.  
   
   By Lemma~\ref{lemma:inf_ord_elt}, there exists an element $q \in Q = G/H$ of infinite order which induces an infinite-order outer automorphism $\phi_q \in \Out(H)$. Also, by Lemma~\ref{lemma_one_end}, there exists an $m \in \N^+$ such that $\phi_q^m([H_v]) = [H_v]$ for the one-ended vertex group $H_v$. 
   
   Suppose the subgroup $H_v$ is hyperbolic. By Theorem~\ref{thm:bow_jsj}, there is a canonical JSJ decomposition for $H_v$. Suppose towards a contradiction that the JSJ decomposition for $H_v$ contains a vertex group $H_v^i$ of type~(3). We prove there exists $n \in \N^+$ such that $\phi_q^n$ restricts to an infinite-order element of $\Out(H_v^i)$ to produce a contradiction to Theorem~\ref{thm:hyp_rig_fin}. As explained in Section~\ref{sec:grushko_jsj}, there exists $\ell \in \N^+$ so that the restriction $\phi_q^{m\ell}|_{H_v} \in \Out(H_v)$ preserves the vertex groups in the JSJ decomposition of $H_v$, yielding a further restriction to the rigid vertex group $H_v^i$ of $H_v$, an element $\psi = \phi_q^{m\ell}|_{H_v^i} \in \Out(H_v^i)$. Thus $H_v^i \rtimes_{\psi} \Z$ is a subgroup of $G$. Since $H_v^i$ is not isomorphic to $\Z$, this outer automorphism is infinite by Lemma~\ref{lemma:inf_ord_elt}(1) and Lemma~\ref{lemma:prod_cat0}, contradicting Theorem~\ref{thm:hyp_rig_fin}. 
   
   If the subgroup $H_v$ is $\CAT(0)$ with isolated flats, then analogous arguments, using Theorem~\ref{thm:jsj_gl} and  Theorem~\ref{thm:cat_rig_fin}, prove $H_v$ has no vertex groups of type~(0.a). 
  \end{proof}
   
  \begin{example} \label{example:normal_sub}
    If $Y_f$ is the mapping torus of a homeomorphism $f: Y \rightarrow Y$ of a space $Y$, then $\pi_1(Y) \triangleleft \pi_1(Y_f)$ is an infinite-index normal subgroup. The following list provides examples of normal subgroups of various types as in Theorem~\ref{thm_structure_normal}.
    \begin{enumerate}
     \item[(a)] Let $\Sigma$ be a connected surface with non-empty boundary and negative Euler characteristic. Let $\Psi:\Sigma \rightarrow \Sigma$ be a pseudo-Anosov homeomorphism of $\Sigma$ that fixes the boundary components of $\Sigma$ pointwise, and let $Y_{\Psi}$ denote the corresponding mapping torus. By work of Thurston~\cite{thurston,otal}, $\pi_1(Y_{\Psi})$ acts by isometries on hyperbolic $3$-space with quotient a finite-volume hyperbolic $3$-manifold $M$ with torus cusps. Let $M_\Psi$ denote the manifold $M$ truncated at its cusps so $\pi_1(M_\Psi) \cong \pi_1(Y_{\Psi})$. By \cite[Theorem II.11.27]{bridsonhaefliger} and \cite[Proposition~9.1]{hruska}, the universal cover $\widetilde{M_\Psi}$ is $\CAT(0)$ with isolated flats. Ruane \cite{ruane} proved the $\CAT(0)$ boundary of $\widetilde{M_\Psi}$ is homeomorphic to the Sierpinski carpet.     
     \item[(b)] Suppose $\Sigma_1, \ldots, \Sigma_k$ and $\Psi_1, \ldots, \Psi_k$ are surfaces and pseudo-Anosov maps as in (a) for $k \geq 2$. Let $Y$ be the union of the surfaces $\Sigma_i$ glued along their boundary components by homeomorphisms in some configuration so that each boundary component is glued to at least one other, and let $\Psi:Y \rightarrow Y$ be a homeomorphism so that $\Psi|_{\Sigma_i} = \Psi_i$. Then the mapping torus $Y_{\Psi}$ is the union of truncated $3$-manifolds $M_{\Psi_1}, \ldots, M_{\Psi_k}$ glued together along their boundary tori. As above, the universal cover of $Y_{\Psi}$ is $\CAT(0)$ with isolated flats. Depending on $k$ and the gluing configuration, the space $Y$ is either a closed surface of negative Euler characteristic or a one-ended hyperbolic group whose JSJ decomposition does not contain any rigid vertex groups. 
     \item[(c)] Let $\Sigma$ be a closed connected surface of negative Euler characteristic, and let $f_\Sigma:\Sigma \rightarrow \Sigma$ be a pseudo-Anosov homeomorphism that fixes a point $x \in \Sigma$. Let $T$ denote the $k$-torus, and let $f_T:T \rightarrow T$ be the identity map. Let $Y$ be the wedge of $\Sigma$ and $T$, where $x \in \Sigma$ is identified to a point on $T$, and let $f: Y \rightarrow Y$ be a homeomorphism so that $f|_{\Sigma} = f_{\Sigma}$ and $f|_T = f_T$. Then the universal cover of the mapping torus is $\CAT(0)$ with isolated flats, and $\pi_1(Y) = \pi_1(\Sigma) * \Z^k$ is a normal $\CAT(0)$ subgroup with isolated flats. See Figure~\ref{figure:noCT}. 
    \end{enumerate}
  \end{example}
  
  \begin{example}
   The following example does not arise from a mapping torus of a homeomorphism of a $2$-complex. Let $H = \cA * \cB * \pi_1(\Sigma)$, where $\cA\cong \Z^k$ and $\cB \cong \Z^\ell$ with $k, \ell \geq 2$, and $\Sigma$ is a closed hyperbolic surface. The group $H$ is $\CAT(0)$ with isolated flats. Let $w \in \pi_1(\Sigma)$ be a non-trivial element and $\Phi$ an automorphism of $\pi_1(\Sigma)$ induced by a pseudo-Anosov homeomorphism of $\Sigma$. Let 
   \[ G= \Biggl\langle 
	\begin{array}{c|cc}
	    &	 tat^{-1} = a & \forall a \in \cA	\\
	  \cA, \cB, \pi_1(\Sigma), t  &	tbt^{-1}=w^{-1}bw 		&\forall b \in \cB	\\
	    &	tct^{-1} = \Phi(c) & \forall c \in \pi_1(\Sigma) 	\\
	\end{array}
      \Biggr\rangle. \]
  Since $G  = \la \cA, t \ra *_{\la t \ra} \Big( \la \pi_1(\Sigma), wt \ra *_{\la wt \ra} \la \cB, wt \ra  \Big)$, the group $G$ is $\CAT(0)$ by \cite[Proposition II.11.17]{bridsonhaefliger}. The group $G$ is the fundamental group of a space obtained by gluing a $(k+1)$-torus $T_k$ and an $(\ell+1)$-torus $T_{\ell}$ to a closed hyperbolic $3$-manifold $M_{\Phi}$ so that essential simple closed curves on $T_k$ and $T_{\ell}$ are glued to disjoint geodesic essential simple closed curves in $M_{\Phi}$. Thus, the group $G$ is $\CAT(0)$ with isolated flats. 
  \end{example}

  \section{The outer automorphism group of a hyperbolic group}
  
   \begin{defn}
    Let $\G$ be a group. An outer automorphism $\phi \in \Out(\G)$ is {\it atoroidal} if no power of $\phi$ preserves the conjugacy class of any infinite-order element in $\G$. A subgroup $K \leq \Out(\G)$ is {\it purely atoroidal} if every non-trivial element of $K$ is atoroidal.
  \end{defn}
 
  \begin{prop}\label{prop:pur_ator_Fn}
  There does not exist a purely atoroidal subgroup of $\Out(F_r)$ isomorphic to $\Z^2$. 
  \end{prop}
  
  \begin{remark} Feighn--Handel~\cite{fh09} classify abelian subgroups of $\Out(F_r)$ up to finite index, and our proof of Proposition~\ref{prop:pur_ator_Fn} below relies on their analysis. Given an element $\phi \in \Out(F_r)$, Feighn--Handel define an abelian subgroup $D(\phi)$ called the {\it disintegration subgroup} of $\phi$; an example is given in this remark. They prove that for every abelian subgroup $A \leq \Out(F_r)$ there exists $\phi \in A$ so that $D(\phi) \cap A$ has finite index in $A$. The techniques in the following proof are not used elsewhere in this paper, so we refer the reader to the work of Feighn--Handel~\cite{fh09} for additional background, motivation, and definitions. 
   
  The disintegration subgroup of an outer automorphism of the free group is motivated by the disintegration of a mapping class group element. If $\psi$ is an element of the mapping class group of a surface $\Sigma$, then by Thurston's classification theorem there exists a corresponding decomposition of the surface $\Sigma$ into a disjoint union of subsurfaces $\Sigma_i$, which either are annuli or have negative Euler characteristic. There exists a homeomorphism $h \colon \Sigma \to \Sigma$ representing $\psi$ such that for each annulus, $h|_{\Sigma_i}$ is a non-trivial Dehn twist, and for the other subsurfaces, $h|_{\Sigma_i}$ is either a pseudo-Anosov homeomorphism or the identity. Suppose the subsurfaces are numbered such that $h|_{\Sigma_i}$ is not the identity for $1 \leq i \leq n$. For $1 \leq i \leq n$, let $h_i \colon \Sigma \to \Sigma$ agree with $h$ on $\Sigma_i$ and be the identity on $\Sigma - \Sigma_i$. The disintegration of $\psi$ is the abelian group generated by $\{h_i \, | \, 1 \leq i \leq n\}$. 

  A (rotationless) outer automorphism has a normal form given by a topological representative $g \colon G \to G$ called a completely split train track map of a finite graph $G$. The graph $G$ decomposes as a disjoint union of subgraphs called strata, and $g$ restricted to a non-zero stratum is either exponentially growing, non-exponentially growing, or the identity. However, unlike surface homeomorphisms, these strata might not be invariant under $g$. The disintegration process determines certain invariant subgraphs of $G$ and defines an abelian group $D(\phi)$. For example, let $\phi$ be the atoroidal outer automorphism with a completely split train track representative on a 6-petaled rose with edges labeled $\{a,b,c,d,e,f\}$ given by $$g: a \mapsto ac , b \mapsto a, c \mapsto b, d \mapsto df, e \mapsto d, f \mapsto e. $$ The group $D(\phi)$ is isomorphic to $\Z^2$ and is generated by $$g_1: a \mapsto ac , b \mapsto a, c \mapsto b, d \mapsto d, e \mapsto e, f \mapsto f,$$ $$g_2: a \mapsto a , b \mapsto b, c\mapsto c, d \mapsto df, e \mapsto d, f \mapsto e.$$ In particular, $D(\phi)$ is not purely atoroidal. 
  \end{remark}

  \begin{proof}[Proof of Proposition~\ref{prop:pur_ator_Fn}]
  Assume towards a contradiction $A \cong \Z^2$ is a purely atoroidal subgroup of $\Out(F_r)$. By \cite[Theorem 7.2]{fh09}, there exists an element $\phi \in A$ such that $D(\phi) \cap A$ is a finite-index subgroup of $A$,  where $D(\phi)$ is the disintegration subgroup of $\phi$. We will analyze the subgroup $D(\phi)$ to obtain an element $\psi \in D(\phi) \cap A$ that is not purely atoroidal. 
  
  Let $f \colon G \to G$ be a completely split train track representative of $\phi$. Then $G$ has a filtration $G_1 \subset G_2 \subset \ldots \subset G_K$, where each filtration element $G_i$ is an $f$-invariant subgraph for $1 \leq i \leq K$. Let the stratum $H_i$ be the closure of $G_i - G_{i-1}$. The element $\phi$ is atoroidal, so the graph $G$ has no closed indivisible Nielsen paths, (quasi)-exceptional paths, or non-fixed non-exponentially growing edges. Thus, each stratum has one of three types: (i) exponentially growing, (ii) fixed non-exponentially growing, or (iii) zero stratum. 
  
  The B-graph used to define $D(\phi)$ in this setting has vertex set in one-to-one correspondence with the exponentially growing stratum of $G$. There is an edge $\{H_i, H_j\}$ in the B-graph for $i >j$ if there exists an edge $e_i \subset H_i$ so that $f(e_i)$ intersects $H_j$ in a certain way; see~\cite[Definition 6.3]{fh09}. By definition of $D(\phi)$, the rank of $D(\phi)$ gives a lower bound to the number of components of the $B$-graph. Thus, the B-graph has components $B_1, \ldots, B_k$ with $k \geq 2$. Let $B_1$ denote the component of the B-graph containing $H_r$, the exponentially growing stratum in $G$ with the smallest value of $r$. Let $X_1 \subset G$ be the union of the exponentially growing stratum in $B_1$ and their associated zero stratum. Let $X_1' = X_1 \cup G_r$. Since $H_r$ is exponentially growing, $H_r$ is connected and $G_r$ is not a forest. Hence, the rank of $\pi_1(X_1')$ is at least one. 
  
  Since the subgraph $X_1$ contains an exponentially growing stratum, there exists a nontrivial lamination $\Lambda$ associated to $X_1$. The lamination $\Lambda$ is stabilized by $D(\phi)$, and there is a corresponding Perron--Frobenius stretch factor homomorphism $PF_{\Lambda} \colon D(\phi) \cap A \to \integers$. Since $D(\phi) \cap A \cong \Z^2$, the kernel has rank at least one, so there exists a non-trivial element $\psi \in D(\phi) \cap A$ that does not stretch $\Lambda$. By definition of $D(\phi)$, the element $\psi$ is the identity on $X_1$. Moreover, since $X_1'$ is the union of $X_1$ and a collection of fixed non-exponentially growing edges, the element $\psi$ fixes $X_1'$. Since $X_1'$ has rank at least one, the element $\psi$ is not atoroidal, a contradiction. 
    \end{proof}  
  
  \begin{cor} \label{cor:pur_tor_vFn} 
      If $H$ is a finitely generated virtually free group, then there does not exist a purely atoroidal subgroup of $\Out(H)$ isomorphic to $\Z^2$. 
  \end{cor}
  \begin{proof}
    Suppose there exists a purely atoroidal subgroup $A = \la \phi, \psi \ra \cong \Z^2 \subset \Out(H)$. We will use this subgroup to construct a purely atoroidal $\Z^2$ subgroup in the outer automorphism group of a free group, contradicting Proposition~\ref{prop:pur_ator_Fn}.

    Let $F_r \leq H$ be a free subgroup of finite index $N$ in $H$. Since $H$ is finitely generated, there are finitely many subgroups of a given index. Therefore, we can take $F_r$ to be a characteristic subgroup of $H$ (for instance, the intersection of all free subgroups of a given finite index). We fix a choice $\Phi,\Psi \in \operatorname{Aut}(H)$ of representatives of $\phi$ and $\psi$, respectively. Since $F_r$ is characteristic we can take the restrictions $\Phi' = \Phi|_{F_r}$, $\Psi' = \Psi|_{F_r} \in \Aut(F_r)$ with images $\phi', \psi' \in \Out(F_r)$.

    We claim there exists $M > 0$ such that $[\phi'^M, \psi'] = 1 \in \Out(F_r)$.
    Since $\phi$ and $\psi$ are commuting elements of $\Out(H)$, there exists $h \in H$ such that $[\Phi, \Psi] = i_h$.
    We compute that $[\Phi^n, \Psi] = i_{h_n}$ for $h_n = \Phi^{n-1}(h) \dots \Phi(h) h$.
    Now we need to show that for some $M > 0$ we have $h_M \in F_r$.
    Note that $h_{n+1} = \Phi(h_n) h$.
    Let $\overline{h}_n$ be the image of $h_n$ in $H / F_r$.
    Since $F_r$ is characteristic, $\Phi$ induces $\overline{\Phi} \in \Aut(H / F_r)$ and $\overline{h}_{n+1} = \overline{\Phi}(\overline{h}_n) \overline h$.
    The mapping $g \mapsto \overline{\Phi}(g) h$ is a bijection of the finite set $H / F_r$, so it is periodic.
    Since $\overline{h}_0 = 1$, there exists $M > 0$ with $\overline{h}_M = 1$, that is, $h_M \in F_r$.

    Thus, after replacing $\phi$ with $\phi^M$ if necessary, we have a homomorphism $A \to \Out(F_r)$ defined by $\phi \mapsto \phi', \psi \mapsto \psi'$, which sends every outer automorphism $\theta$ to the outer automorphism $\theta'$ given by restricting some lift of $\theta$.
    Every non-trivial $\theta \in A$ is atoroidal in $\Out(H)$ and therefore $\theta'$ is atoroidal in $\Out(F_r)$ (since conjugacy in $F_r$ is a finer equivalence relation than conjugacy in $H$).
    This gives both that the homomorphism $A \to \Out(F_r)$ is injective and that its $\Z^2$ image is purely atoroidal.
  \end{proof}
  
  \begin{thm}\label{thm:pur_ator}
    Let $H$ be a hyperbolic group. Then there does not exist a purely atoroidal subgroup of $\Out(H)$ isomorphic to $\Z^2$. 
  \end{thm}
  \begin{proof}
    Without loss of generality, we may assume $H$ is infinite, and not virtually $\Z$. Suppose there exists a purely atoroidal subgroup $A = \la \phi, \psi \ra \cong \Z^2 \leq \Out(H)$.
  
    Suppose first that the group $H$ is one-ended. Since centralizers of pseudo-Anosov homeomorphisms, which are also the atoroidal homeomorphisms of a surface, are virtually cyclic \cite{mccarthy}, we conclude by an analogous argument to that in the proof of Corollary~\ref{cor:pur_tor_vFn} that the group $H$ is not Fuchsian. By \cite{paulin}, the group $H$ splits over a $2$-ended subgroup since the outer automorphism group of $H$ is infinite. Thus, the group  $H$ has a non-trivial JSJ decomposition containing a $2$-ended vertex group $H_v$. Up to passing to a finite-index subgroup, the group $A$ preserves $[H_v]$, a contradiction. 
 
    Now suppose $H$ is infinite-ended. Let $\cH$ be a Stallings--Dunwoody splitting of $H$ over finite groups. Since $H$ is not virtually free by Corollary~\ref{cor:pur_tor_vFn}, there exists a one-ended vertex group $H_v$ of $\cH$. By Lemma~\ref{lemma_one_end}, there exists $m, n \in \N^+$ such that $\phi^m([H_v]) = [H_v]$ and $\psi^n([H_v]) = [H_v]$. Let $N=mn$. We first claim that $\la \phi^N|_{H_v}, \psi^N|_{H_v} \ra$ is an abelian subgroup of $\Out(H_v)$. To see this, let $\Psi, \Phi \in \Aut(H)$ be representatives of $\psi$ and $\phi$, respectively, such that $\Psi^N(H_v) = H_v$, $\Phi^N(H_v) = H_v$, and the commutator $[\Phi^N, \Psi^N]$ is an inner automorphism given by some $h \in H$. Then $[\Phi^N, \Psi^N](H_v) = H_v = hH_vh^{-1}$.  The subgroup $H_v$ fixes a unique vertex $v$ of the Bass--Serre tree corresponding to $\cH$, and $hH_vh^{-1}$ fixes a unique vertex $h\cdot v$. Thus, $v = h\cdot v$, so $h \in H_v$. Therefore, $\phi^N$ and $\psi^N$ restrict to commuting outer automorphisms of $H_v$.  Now, $\la \phi^N|_{H_v}, \psi^N|_{H_v} \ra$ is a purely atoroidal subgroup of $\Out(H_v)$ isomorphic to $\Z^2$. The argument from the previous paragraph yields a contradiction. 
  \end{proof}

    \section{The nonexistence of Cannon--Thurston maps}
    
\subsection{Hyperbolic normal subgroups}    

    To prove that Cannon--Thurston maps do not exist for hyperbolic normal subgroups of a $\CAT(0)$ group with isolated flats, we use the following criterion.
  \begin{lemma} \label{lemma:crit_nonexist}
   Let $H \leq G$ such that $H$ acts geometrically on $Y$ and $G$ acts geometrically on $X$, where $X$ and $Y$ are hyperbolic or $\CAT(0)$ spaces. If there exists an element $h \in H$ so that 
   \begin{enumerate}
    \item $h$ acts by North-South dynamics on $\p Y$;
    \item $h$ acts by translation on a flat $F \subset X$; \emph{and}
    \item  $\p F \subset \Lambda_xH \subset \p X$,
   \end{enumerate}
      then the Cannon--Thurston map does not exist for $(H, G)$.
  \end{lemma}
  \begin{proof}
   Suppose there exists such an element $h$ and let its fixed points on $\p Y$ be denoted $h^{\pm \infty}$. Suppose towards a contradiction that the Cannon--Thurston map $\hi:\p Y \rightarrow \p X$ exists. Since the Cannon--Thurston map is $H$-equivariant (Remark~\ref{rem:H_equivariant}), $\Lambda_xH \subset \hi(\p Y)$. By Lemma~\ref{lemma:fp_to_fp}, the element $h$ acts by North-South dynamics on $\hi(\p Y)$ fixing the points $\hi(h^{\pm \infty})$. However, $h$ acts by translations on the flat $F$; hence, $h$ acts trivially on $\p F \subset \Lambda_xH \subset \hi(\p Y)$, a contradiction.
  \end{proof}
  
  \begin{prop} \label{prop_elt_H_fstab}
    Let $G$ be a $\CAT(0)$ group with isolated flats, and let $H \triangleleft G$ be a hyperbolic normal subgroup. Then there exists an infinite order element $h \in H$ which acts by translation on a flat under any geometric action of $G$ on a $\CAT(0)$ space $X$.
  \end{prop}
  \begin{proof}
    Consider any subgroup $A \leq G$ with $A \cong \Z^2$; such a subgroup exists by our definition of isolated flats requiring the set of flats to be non-empty together with \cite[Lemma 3.1.2]{hruskakleiner}, a converse to the Flat Torus Theorem. We claim that some infinite order $h \in H$ lies in a $\Z^2$ subgroup of $G$, which implies the proposition by the standard Flat Torus Theorem. If $A \cap H \neq 1$ then we are done, so suppose $A \cap H = 1$. Consider the map $A \to \Out(H)$ coming from conjugation in $G$. If it has kernel, then there exists some non-trivial $g \in A$ and some $h_0 \in H$ such that $h^g = h^{h_0}$ for all $h \in H$. If not, then its image is isomorphic to $\Z^2$ and thus cannot be purely atoroidal by Theorem~\ref{thm:pur_ator}, so there is some non-trivial $g \in A$ and some infinite order $h \in H$ such that $h^g = h^{h_0}$ for some $h_0 \in H$. In either case there is some infinite order $h \in H$ commuting with $g h_0^{-1}$. Since $A \cap H = 1$ we see that $g$ has infinite order in the quotient $G/H$ and thus no non-zero power of $g h_0^{-1}$ is in $H$, so $\langle h, g h_0^{-1} \rangle \cong \Z^2$.
  \end{proof}
      
  \begin{thm} \label{thm_noCT_hyp}
    Let $G$ be a $\CAT(0)$ group with isolated flats, and let $H \triangleleft G$ be an infinite hyperbolic normal subgroup.    Then the Cannon--Thurston map does not exist for $(H, G)$. 
  \end{thm}
  \begin{proof}
    By Proposition~\ref{prop_elt_H_fstab}, there exists an infinite-order element $h \in H$ which acts by translation on a flat in $X$. Since $H$ is hyperbolic, the element $h$ acts by North-South dynamics on $\p H$ by Theorem~\ref{thm_hyp_NS}. By Lemma~\ref{lemma_cat_ls_normal}, $\hi(\p H) = \p G \cong \p X$. Thus, the hypotheses of Lemma~\ref{lemma:crit_nonexist} are satisfied, and the Cannon--Thurston map does not exist.  
  \end{proof}
  
  The above results imply that the Cannon--Thurston map for the pair $(H,G)$ also does not exist with respect to their Morse boundaries, $\p_MH$ and $\p_MG$. In this setting, we define the Cannon--Thurston map as before.
  
  \begin{defn}
      Let $H \leq G$ so that $H$ acts geometrically on $Y$ and $G$ acts geometrically on $X$, where $X$ and $Y$ are proper geodesic metric spaces.  Let $x \in X$ and $y \in Y$. A \emph{Cannon--Thurston map} for the pair $(H, G)$ with respect to their Morse boundaries is the map $f \colon \p_M Y \to \p_M X$ defined by \[f\left(\lim_{i \rightarrow \infty} h_i \cdot y\right)=\lim_{i \rightarrow \infty} h_i \cdot x,\] provided $f$ is well-defined and continuous. 
  \end{defn}
  
  \begin{prop} \label{prop:Morse}
   Let $G$ be a $\CAT(0)$ group with isolated flats, and let $H \triangleleft G$ be an infinite hyperbolic normal subgroup.    Then the Cannon--Thurston map does not exist for $(H, G)$ with respect to their Morse boundaries. 
  \end{prop}
  \begin{proof}
    Suppose $H$ acts geometrically on a proper geodesic metric space $Y$, and suppose the group $G$ acts geometrically on a $\CAT(0)$ space $X$. There exists an infinite-order element $h \in H$ that acts by translation on a flat in $X$ by Proposition~\ref{prop_elt_H_fstab}. Let $y \in Y$, and let $x \in X$ be a point in the aforementioned flat stabilized by $h$. Since $H$ is hyperbolic and $h \in H$ has infinite-order, the sequence $\{h^n \cdot y\}_{n \in \N}$ converges to a point in $\p H \cong \p_MH$. However, the sequence $\{h^n \cdot x\}_{n \in \N}$ converges to a point in the boundary of a flat in $X$, and hence does not converge to a point in $\p_MG$. Thus, the Cannon--Thurston map does not exist for the pair $(H,G)$ with respect to their Morse boundaries. 
  \end{proof}

  \subsection{Normal \texorpdfstring{$\CAT(0)$}{CAT(0)} subgroups with isolated flats}

  \begin{thm} \label{thm_noCT_IFP}
   Let $H \triangleleft G$ be $\CAT(0)$ groups with isolated flats and suppose that $H$ is an infinite, infinite-index normal subgroup of $G$. Then the Cannon--Thurston map does not exist for $(H,G)$.
  \end{thm}
  \begin{proof}
    Suppose $H$ and $G$ act geometrically on $Y$ and $X$, respectively, where $Y$ and $X$ are $\CAT(0)$ spaces with isolated flats. The structure of the normal subgroup $H$ is guaranteed by Theorem~\ref{thm_structure_normal}. By Lemma~\ref{lemma:crit_nonexist} and Lemma~\ref{cor:char_isoms}, we may assume that the only infinite-order elements of $H$ that stabilize a flat in $X$ lie in a conjugate of a (peripheral) free abelian subgroup of $H$ of rank at least two. Let $y \in Y$ and $x \in X$. To prove the Cannon--Thurston map does not exist, we will exhibit sequences $\{h_n\}_{n \in \N}, \{h'_n\}_{n \in \N} \subset H$ so that 
     \begin{eqnarray}
       \lim_{n \rightarrow \infty} h_n \cdot y &=& \lim_{n \rightarrow \infty} h'_n \cdot y \in \p Y. \\
       \lim_{n \rightarrow \infty}  h_n \cdot x &\neq& \lim_{n \rightarrow \infty} h'_n \cdot x \in \p X.     
     \end{eqnarray}

   \begin{figure}
   \begin{overpic}[width=.95\textwidth,tics=5]{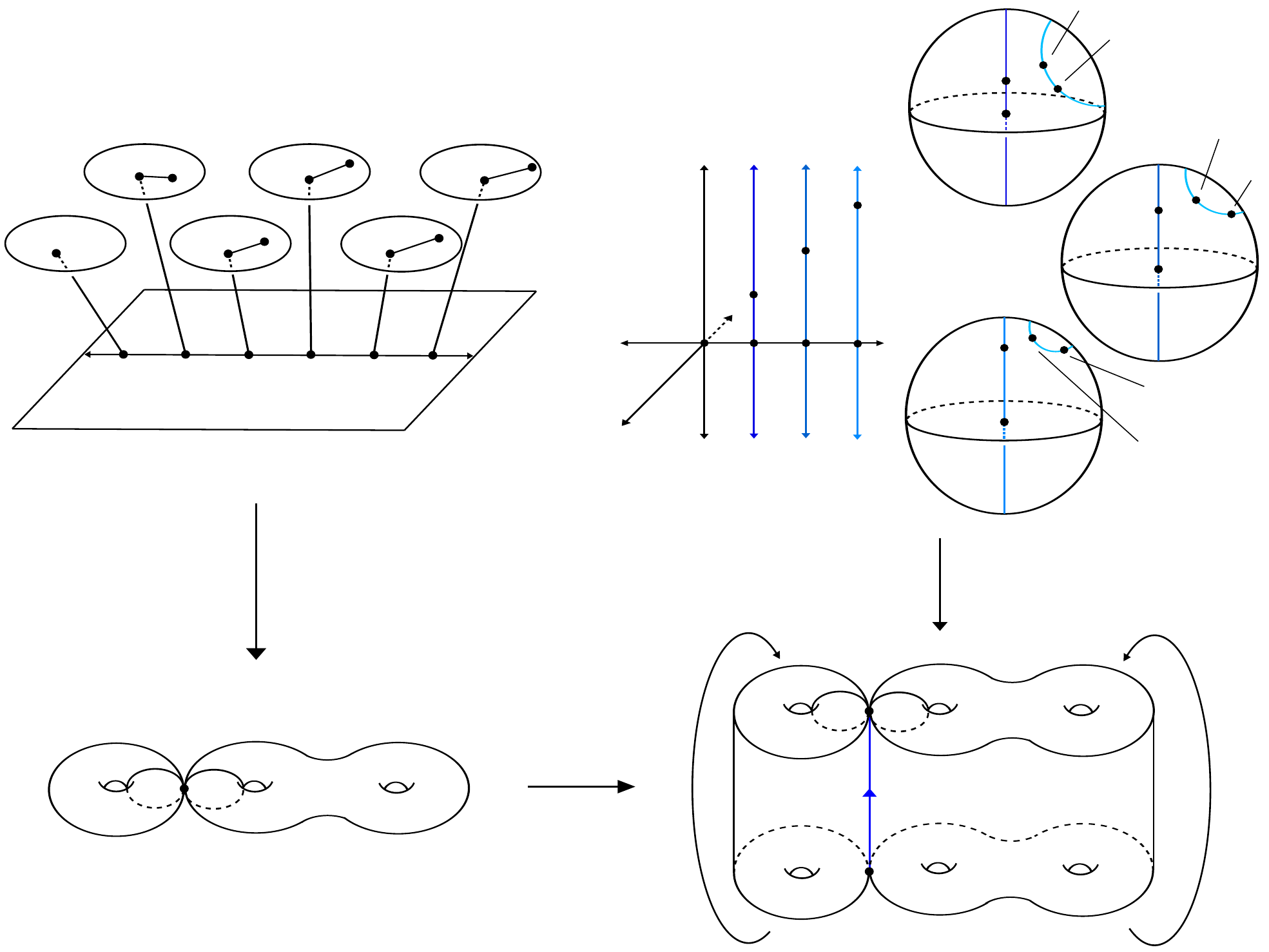} 
   \put(2,16){\small{$\overline{Y}$}}
   \put(25,12){\small{$\Sigma $}}
   \put(11.5,12){\small{$a$}}
   \put(16,12){\small{$b$}}
   \put(1,44){\small{$\E^2$}}
   \put(3,60){\small{$\Hy^2$}}
   \put(38,46){\small{$\alpha$}}
   \put(2,66){\small{$Y$}}
   \put(14,44.5){\small{$a$}}
   \put(18.5,44.5){\small{$a^2$}}
   \put(23.5,44.5){\small{$a^3$}}
   \put(28.5,44.5){\small{$a^4$}}
   \put(33,44.5){\small{$a^5$}}
   \put(11,61.7){\tiny{$a\phi(b)$}}
   \put(27,63.5){\tiny{$a^3\phi^3(b)$}}
   \put(42,63){\tiny{$a^5\phi^5(b)$}}
   \put(60,70){\small{$X$}}
   \put(52,20){\small{$\overline{X}$}}
   \put(62,12){\small{$T^3$}}
   \put(52,10){\small{$id$}}
   \put(96,10){\small{$\phi$}}
   \put(78,12){\small{$M_{\phi}$}}
   \put(69.5,12){\small{$t$}}
   \put(65.5,18.3){\small{$a$}}
   \put(70.3,18){\small{$b$}}
   \put(50,55){\small{$\E^3$}}
   \put(79,54){\small{$\Hy^3$}}
   \put(60,48.5){\tiny{$a$}}
   \put(63.8,48.5){\tiny{$a^2$}}
   \put(68,48.5){\tiny{$a^3$}}
   \put(60,51.5){\tiny{$at$}}
   \put(63.9,55){\tiny{$a^2t^2$}}
   \put(68,58.5){\tiny{$a^3t^3$}}
   \put(69,46){\small{$a$}}
   \put(55,62.5){\small{$t$}}
   \put(58.5,63){\tiny{$\gamma_1$}}
   \put(62.5,63){\tiny{$\gamma_2$}}
   \put(66.7,63){\tiny{$\gamma_3$}}
   \put(78.5,75){\tiny{$\gamma_1$}}
   \put(90.5,63){\tiny{$\gamma_2$}}
   \put(78,51){\tiny{$\gamma_3$}}
   \put(77.5,65.5){\tiny{$a$}}
   \put(89,52.8){\tiny{$a^2$}}
   \put(76.7,41){\tiny{$a^3$}}
   \put(77,68.3){\tiny{$at$}}
   \put(87.5,58){\tiny{$a^2t^2$}}
   \put(75.3,47){\tiny{$a^3t^3$}}
   \put(84.5,74.5){\tiny{$atb$}}
   \put(87.5,72){\tiny{$atbt^{-1}$}}  
   \put(94,64.5){\tiny{$a^2t^2b$}}
   \put(97,61){\tiny{$a^2t^2bt^{-2}$}}   
   \put(90.5,44){\tiny{$a^3t^3bt^{-3}$}} 
   \put(90,39.5){\tiny{$a^3t^3b$}}   
   \end{overpic}
    \caption{On the left, the universal cover of the wedge of a torus and the surface~$\Sigma$ is homotopic to $Y$, the union of countably many copies of the Euclidean and hyperbolic planes attached to each other along line segments. On the right, the space $\overline{X}$ is the union of a $3$-torus and the mapping torus of a pseudo-Anosov homeomorphism $\phi:\Sigma \rightarrow \Sigma$  glued to each other along a circle. The universal cover $X$ is the union of countably many copies of $\E^3$ and $\Hy^3$ glued to each other along lines. The groups $H = \pi_1(\overline{Y})$ and $G = \pi_1(\overline{X})$ are $\CAT(0)$ with isolated flats. The Cannon--Thurston map from $\p H$ to $\p G$ does not exist because the orbit points corresponding to the sequences $\{a^n\}$ and $\{a^n\phi_*^n(b) = a^nt^nbt^{-n}\}$ converge to the same point $\alpha \in \p Y \cong \p H$, but converge to different points in $\p X \cong \p G$. Indeed, the geodesic in $X$ from the basepoint to $a^nt^nbt^{-n}$ passes through the $R$-neighborhood of $a^nt^n$ for some $R$ independent of $n$.}
    \label{figure:noCT}
    \end{figure}
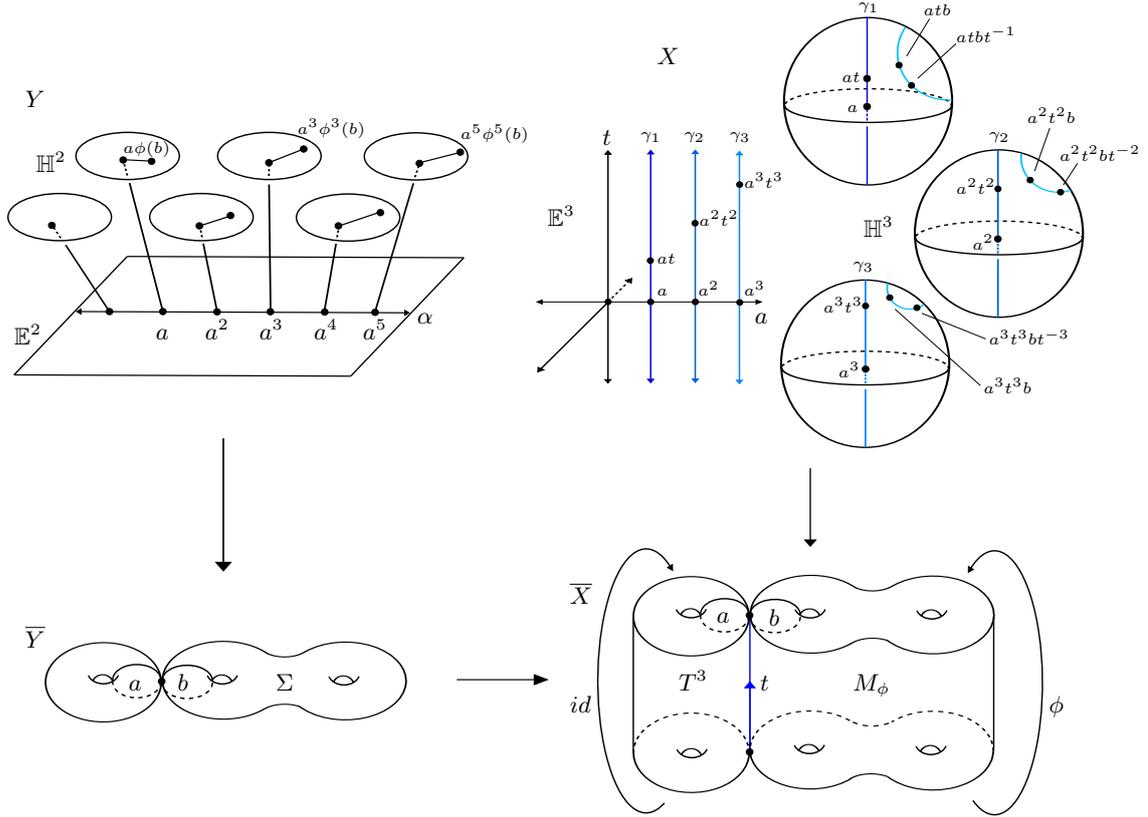

    Let $\cH$ be a Stallings--Dunwoody splitting of $H$. Since $H$ is not virtually free, there exists a one-ended vertex group in $\cH$. 
    
   \noindent {\bf Case 1: $\cH$ has a one-ended vertex group which is a virtually abelian subgroup.}
    
   We first define the sequences $\{h_n\}_{n \in \N},\{h_n'\}_{n \in \N} \subset H$. An example appears in Figure~\ref{figure:noCT}. Let $\cA' \leq H$ denote a one-ended vertex group of $\cH$ that is virtually abelian. Let $\cA \leq \cA'$ be a finite-index subgroup isomorphic to $\Z^k$ for some $k \geq 2$. By Lemma~\ref{lemma:inf_ord_elt}, there is an infinite-order element $q \in Q = G/H$ that induces an infinite-order outer automorphism $\phi_q \in \Out(H)$. By Lemma~\ref{lemma_one_end}, there exists $m \in \N^+$ such that $\phi_q^m([\cA']) = [\cA']$. Up to passing to a further power if necessary, we may assume that $\phi_q^m([\cA]) = [\cA]$. By Lemma~\ref{lemma:inf_ord_elt}, the group $G$ contains a subgroup isomorphic to the group $\cA \rtimes_{\phi_q^m} \Z$, which is solvable. By the Solvable Subgroup Theorem \cite[Theorem II.7.8]{bridsonhaefliger}, this solvable subgroup is virtually abelian. Thus, after possibly passing to a power, the element $\phi_q^m$ restricted to $\cA$ is the identity. Therefore, there exists $t \in G$ with $p(t) = q^m$, where $p:G \rightarrow Q$ is the quotient map, and so that $\phi_q^{mn}(h) = t^nht^{-n} = h$ for all $h \in \cA$ and $n \in \N$.
   
   Since the group $H$ is not virtually abelian, there exists an infinite order element $h_0 \in H \setminus \cA'$. The subgroup $\la \cA', h_0 \ra \leq H$ is also not virtually abelian. Indeed, if $\la \cA', h_0 \ra$ were virtually abelian, and hence one-ended, then this group would act elliptically on any tree on which it acts with finite edge stabilizers including in particular the Bass--Serre tree for $\cH$, on which it would fix the same vertex as $\cA'$, a contradiction. We claim that $h_0$ does not commute with any power of $t$. Indeed, in this case, $\la \cA', h_0 \ra \oplus \la t^k \ra$ would be a subgroup of $G$, contradicting Lemma~\ref{lemma:prod_cat0}. 

  Consider the elements $t^nh_0t^{-n}$ for $n \in \N$. Fix a point $p$ in the flat stabilized by $\cA$ in $Y$. Choose $a_n \in \cA$ so that $d(a_n\cdot p, t^nh_0t^{-n}\cdot p)$ is minimized, and among such $a_n \in \cA$, $d(a_n \cdot p, p)$ is maximized. If the length $|a_n|<D$ for all $n \in \N$ and some $D \geq 0$, let $h_1 \in \cA$ be an arbitrary infinite-order element and let $h_n = h_1^n$. Otherwise, after choosing an infinite subsequence (we will abuse notation and continue to write $n \in \N$ rather than $n_k, k \in \N$), we have $|a_n| \rightarrow \infty$ as $n \rightarrow \infty$ so that $\{a_n\cdot p\}_{n \in \N}$ converges to a point~$\alpha$ in the visual boundary of the flat in $Y$ stabilized by $\cA$, by compactness of $Y \cup \p Y$, and $\{a_nt^n \cdot x\}$ converges to a point in the boundary of the flat in $X$ stabilized by $\la \cA, t \ra$. Let $h_n = a_n$. In both cases, let $h_n' = h_nt^nh_0t^{-n}$.
   
   To show (1), suppose $F \subset Y$ is a flat stabilized by $\cA$, and let $y \in F$. Consider the geodesic segment $\gamma_n$ in $Y$ connecting $id \cdot y$ to $h_n'\cdot y = h_nt^nh_0t^{-n} \cdot y = h_na_nw_n \cdot y.$ In the case that $h_n = h_1^n \in \cA$ and the length of $a_n$ is bounded by $D\geq 0$, the geodesic $\gamma_n$ passes through the $D'$-neighborhood of the point $h_1^n\cdot y$ for all $n \in \N$ and some $D' \geq 0$. Thus, $h_n\cdot y$ and $h_n'\cdot y$ converge to the same point in $\p Y$. Otherwise, $h_n = a_n$ and $h_n' = a_na_nw_n$, and, similarly, $h_n\cdot y$ and $h_n'\cdot y$ converge to the same point in $\p Y$.   
   
   To prove (2), we will show $\lim_{n \rightarrow \infty} h'_n \cdot x = \lim_{n \rightarrow \infty} h_nt^n \cdot x$. Suppose the basepoint $x \in X$ lies in the flat $F_0$ stabilized by $\Z^{k+1} \cong \la \cA, t \ra \leq G$. For all $n \in \N$, the points $h_nt^n\cdot x$ lie in the flat $F_0$. Let $F_n \subset X$ denote the flat containing the point $h_nt^nh_0t^{-n} \cdot x$. That is, $F_n = h_nt^nh_0\cdot F_0$. By assumption, $h_0$ does not commute with any non-trivial element of the subgroup $\la \cA,t \ra$. Therefore, in the language of Definition~\ref{def:cat0IFP}, the flats $F_0$ and $F_n$ lie in tubular neighborhoods of distinct flats in $\cF \subset X$. For $S  \subset X$, let $\pi_{F_0}(S)$ denote the {\it almost projection}~\cite{drutusapir,sisto} of $S$ to $F_0$: the set of points of $F_0$ whose distance from $S$ is less than $d(S,F_0) + 1$. The subspace $\pi_{F_0}(F_n)$ has bounded diameter by \cite[Theorem 2.14]{sisto}. For all $n \in \N$ we have $d(h_nt^n\cdot x, \pi_{F_0}(F_n)) = d(h_1t\cdot x, \pi_{F_0}(F_1))$ so in particular these distances are uniformly bounded.  Therefore, by \cite[Lemma 1.15]{sisto}, there exists $R>0$ so that the geodesic connecting $x$ to $h_nt^nh_0t^{-n}\cdot x$ passes through $B_R(h_nt^n\cdot x)$ for all $n \in \N$. Thus, 
   $\lim_{n \rightarrow \infty} h'_n \cdot x = \lim_{n \rightarrow \infty} h_nt^n \cdot x \neq \lim_{n \rightarrow \infty} h_n \cdot x$, as desired.
   
   \vskip.1in
   
   \noindent {\bf Case 2: No vertex group of $\cH$ is isomorphic to a virtually abelian group.}
   
   In this case, there exists $H_1 \leq H$, a one-ended vertex group in the Stallings--Dunwoody decomposition of $H$ so that the elementary JSJ decomposition of $H_1$ contains a maximal virtually abelian vertex group $\cA_H$. In addition, there exists a maximally hanging Fuchsian vertex group $H_v$ adjacent to $\cA_H$. As in Case 1, there exists an infinite-order element $q \in Q$ which yields an infinite-order outer automorphism $\phi_q \in \Out(H)$ so that $\phi_q([\cA_H])= [\cA_H]$ and $\phi_q([H_v]) = [H_v]$. Since $H_v$ is a maximal hanging Fuchsian group, it contains a finite-index subgroup $H_{\Sigma}$ isomorphic to $\pi_1(\Sigma)$, where $\Sigma$ is a surface with negative Euler characteristic and non-empty boundary. After passing to a power, we may also assume $\phi_q([H_{\Sigma}]) = [H_{\Sigma}]$. As above, there exists $t' \in G$ so that $p(t') = q$ and $t'ht'^{-1} = h$ for all $h \in \cA_H$; and, for all $h \in H_v$, there exists $w \in H$ and $h' \in H_v$ so that $t'ht'^{-1} = wh'w^{-1}$. There exists an infinite-order element $\alpha \in \cA_H \cap H_v$ since these vertex groups are adjacent. After taking a power, we may assume that $\alpha \in H_{\Sigma}$. Thus, $w \in \cA_H$, and there exists $t \in G$ so that $p(t) = q$ and $tht^{-1} = h$ for all $h \in \cA_H$, and for all $h \in H_v$, there exists $h' \in H_v$ so that $tht^{-1} = h'$. 
 Let $h_n = \alpha^n$ and let $h_n' = \alpha^nt^nh_0t^{-n}$ for $n \in \N$, where $\alpha$ is as defined above, and $h_0$ is an arbitrary element of $H_{\Sigma}$. 
       
   To prove (1), observe that $t^nh_0t^{-n} = \Phi^n(h_0)$, where we may assume that $\Phi \in \Aut(F_n)$ is a fully irreducible automorphism induced by a pseudo-Anosov homeomorphism of $\Sigma$. Indeed, if $\Phi$ is not induced by a pseudo-Anosov, then $\Phi$ fixes the homotopy class of a curve $\beta$ on $\Sigma$. Then the element $\beta$ stabilizes a flat in $X$ but does not stabilize a flat in $Y$. Hence by Lemma~\ref{cor:char_isoms} and Theorem~\ref{lemma_rk1_NS}, the element $\beta$ acts by North-South dynamics on $Y$. Thus by Lemma~\ref{lemma:crit_nonexist} the Cannon--Thurston map does not exist. So, the automorphism $\Phi$ is induced by a pseudo-Anosov homeomorphism. 
   By \cite[Proposition 1.6]{bestvinafeighnhandel} the sequence $\Phi^n(h_0)$ converges to a leaf $\ell$ in the Bestvina-Feighn-Handel lamination $\Lambda_{\Phi}$, which is a closed, $F_n$-invariant subset of $\partial^2 F_n$. Here $\partial^2 F_n$ is the quotient of $(\partial F_n \times \partial F_n $ minus the diagonal) by the flip action. A lift of $\alpha$ to a Cayley graph of $\F_n$ determines a leaf in $\partial^2 F_n$ corresponding to $\alpha$. By \cite[Proposition 1.8]{bestvinafeighnhandel}, the leaf $\ell$ has bounded overlap with any lift of $\alpha$. This implies there exists a uniform bound on the overlap between any geodesic in $\{[y, \Phi^n(h_0) \cdot y]\}_{n \in \mathbb{N}}$ and any geodesic in $\{[y, \alpha^m \cdot y]\}_{m \in \mathbb{Z}}$. Thus, $\lim_{n \rightarrow \infty} h_n\cdot y = \lim_{n \rightarrow \infty} h_n' \cdot y$.
   
   The proof of (2) is analogous to the proof of (2) in Case~1.
   
  \end{proof}

  \begin{remark}
    Suppose $H \triangleleft G$ are $\CAT(0)$ groups with isolated flats and $G$ acts geometrically on a $\CAT(0)$ space $X$. Suppose, as in the proof above, the only infinite-order elements of $H$ that stabilize a flat in $X$ lie in a conjugate of a (peripheral) free abelian subgroup of $H$ of rank at least two. The proof of Theorem~\ref{thm_noCT_IFP} extends to show there are, in fact, countably many sequences that converge to the same point in $\p H$, but converge to pairwise distinct points in $\p G$. Indeed, as defined in the proof above, consider the sequences $\{h_n\}_{n \in \N}$ and $\{h_{n,k}' = h_nt^{kn}h_0t^{-kn}\}_{n \in \N}$ for a fixed $k \in \N$. Then, for $k \neq k'$, the sequences $\{h'_{n,k}\}$ and $\{h'_{n,k'}\}$ converge to the same point in $\p H$, but converge to rays with different slopes in a flat in $X$, and hence converge to different points in~$\p G$. 
  \end{remark}

\bibliographystyle{alpha}
\bibliography{refs}

\end{document}